\documentclass[11pt,leqno]{article}
\usepackage[margin=1in]{geometry} 
\usepackage{amssymb,amsfonts,amsmath,bbm,mathrsfs,stmaryrd,mathtools}
\usepackage{xcolor}
\usepackage{url}

\usepackage{graphicx}

\usepackage{accents}

\usepackage{extarrows}

\usepackage[shortlabels]{enumitem}
\usepackage{tensor}

\usepackage{xr}
\usepackage[T1]{fontenc}
\usepackage[utf8]{inputenc}

\usepackage[colorlinks,
linkcolor=black!75!red,
citecolor=blue,
pdftitle={},
pdfproducer={pdfLaTeX},
pdfpagemode=UseNone,
bookmarksopen=true,
bookmarksnumbered=true,
backref=page]{hyperref}

\usepackage{tikz}
\usetikzlibrary{arrows,calc,decorations.pathreplacing,decorations.markings,decorations.shapes,intersections,shapes.geometric,through,fit,shapes.symbols,positioning,decorations.pathmorphing}

\makeatletter
\newlength\zig@L
\newlength\zig@La
\newlength\zig@Lb

\newcommand{\xzigrightarrow}[2][]{%
  \mathrel{%
    \settowidth{\zig@La}{$\scriptstyle #2$}%
    \settowidth{\zig@Lb}{$\scriptstyle #1$}%
    \zig@L=\zig@La\relax
    \ifdim\zig@Lb>\zig@L \zig@L=\zig@Lb\fi
    \advance\zig@L by 2.2em\relax
    \tikz[baseline=-0.65ex]{%
      \draw[->,
            line cap=round,
            decorate,
            decoration={zigzag,segment length=4pt,amplitude=1.1pt}]%
        (0,0) -- (\zig@L,0)
        node[midway,above=2pt] {$\scriptstyle #2$}%
        \if\relax\detokenize{#1}\relax\else
          node[midway,below=2pt] {$\scriptstyle #1$}%
        \fi
      ;
    }%
  }%
}
\makeatother

\makeatletter
\newcommand{\squigjoin}{1mu} 

\def\sqleft@{\sim}                    
\def\sqmid@{\sim\mkern-\squigjoin}    

\def\rightsquigarrowfill@{%
  \arrowfill@{\sqleft@}{\sqmid@}{\mkern-4mu\succ}%
}

\newcommand{\xrightsquigarrow}[2][]{%
  \ext@arrow 0359\rightsquigarrowfill@{#1}{#2}%
}
\makeatother

\makeatletter
\newcommand*\circled[1]{\tikz[baseline=(char.base)]{
    \node[shape=circle, draw, inner sep=0pt, 
    minimum height={\f@size},] (char) {\vphantom{WAH1g}#1};}}
\makeatother

\makeatletter
\DeclareRobustCommand\widecheck[1]{{\mathpalette\@widecheck{#1}}}
\def\@widecheck#1#2{%
    \setbox\z@\hbox{\m@th$#1#2$}%
    \setbox\tw@\hbox{\m@th$#1%
       \widehat{%
          \vrule\@width\z@\@height\ht\z@
          \vrule\@height\z@\@width\wd\z@}$}%
    \dp\tw@-\ht\z@
    \@tempdima\ht\z@ \advance\@tempdima2\ht\tw@ \divide\@tempdima\thr@@
    \setbox\tw@\hbox{%
       \raise\@tempdima\hbox{\scalebox{1}[-1]{\lower\@tempdima\box
\tw@}}}%
    {\ooalign{\box\tw@ \cr \box\z@}}}
\makeatother

\usepackage{braket}

\usepackage[amsmath,thmmarks,hyperref]{ntheorem}
\usepackage{cleveref}

\newcommand\nthalias[1]{\AddToHook{env/#1/begin}{\crefalias{lemma}{#1}}}

\nthalias{assumption}
\nthalias{conjecture}
\nthalias{construction}
\nthalias{convention}
\nthalias{corollaryN}
\nthalias{corollary}
\nthalias{definition}
\nthalias{examples}
\nthalias{example}
\nthalias{lemma}
\nthalias{notation}
\nthalias{propositionN}
\nthalias{proposition}
\nthalias{question}
\nthalias{recollection}
\nthalias{remarks}
\nthalias{remark}
\nthalias{sketch}
\nthalias{theoremN}
\nthalias{theorem}

\creflabelformat{enumi}{#2#1#3}

\crefname{section}{Section}{Sections}
\crefformat{section}{#2Section~#1#3} 
\Crefformat{section}{#2Section~#1#3} 

\crefname{subsection}{\S}{\S\S}
\AtBeginDocument{%
  \crefformat{subsection}{#2\S#1#3}%
  \Crefformat{subsection}{#2\S#1#3}%
}

\crefname{subsubsection}{\S}{\S\S}
\AtBeginDocument{%
  \crefformat{subsubsection}{#2\S#1#3}%
  \Crefformat{subsubsection}{#2\S#1#3}%
}

\theoremstyle{plain}

\newtheorem{lemma}{Lemma}[section]
\newtheorem{proposition}[lemma]{Proposition}

\newtheorem{theorem}[lemma]{Theorem}

\theoremstyle{plain}
\theoremnumbering{Alph}
\newtheorem{theoremN}{Theorem}

\theoremstyle{plain}
\theorembodyfont{\upshape}
\theoremsymbol{\ensuremath{\blacklozenge}}

\newtheorem{definition}[lemma]{Definition}

\newtheorem{remark}[lemma]{Remark}

\crefname{definition}{definition}{definitions}
\crefformat{definition}{#2definition~#1#3} 
\Crefformat{definition}{#2Definition~#1#3} 

\crefname{ex}{example}{examples}
\crefformat{example}{#2example~#1#3} 
\Crefformat{example}{#2Example~#1#3} 

\crefname{exs}{example}{examples}
\crefformat{examples}{#2example~#1#3} 
\Crefformat{examples}{#2Example~#1#3} 

\crefname{remark}{remark}{remarks}
\crefformat{remark}{#2remark~#1#3} 
\Crefformat{remark}{#2Remark~#1#3} 

\crefname{remarks}{remark}{remarks}
\crefformat{remarks}{#2remark~#1#3} 
\Crefformat{remarks}{#2Remark~#1#3} 

\crefname{convention}{convention}{conventions}
\crefformat{convention}{#2convention~#1#3} 
\Crefformat{convention}{#2Convention~#1#3} 

\crefname{notation}{notation}{notations}
\crefformat{notation}{#2notation~#1#3} 
\Crefformat{notation}{#2Notation~#1#3} 

\crefname{table}{table}{tables}
\crefformat{table}{#2table~#1#3} 
\Crefformat{table}{#2Table~#1#3}

\crefname{lemma}{lemma}{lemmas}
\crefformat{lemma}{#2lemma~#1#3} 
\Crefformat{lemma}{#2Lemma~#1#3} 

\crefname{proposition}{proposition}{propositions}
\crefformat{proposition}{#2proposition~#1#3} 
\Crefformat{proposition}{#2Proposition~#1#3} 

\crefname{propositionN}{proposition}{propositions}
\crefformat{propositionN}{#2proposition~#1#3} 
\Crefformat{propositionN}{#2Proposition~#1#3} 

\crefname{corollary}{corollary}{corollaries}
\crefformat{corollary}{#2corollary~#1#3} 
\Crefformat{corollary}{#2Corollary~#1#3} 

\crefname{corollaryN}{corollary}{corollaries}
\crefformat{corollaryN}{#2corollary~#1#3} 
\Crefformat{corollaryN}{#2Corollary~#1#3} 

\crefname{theorem}{theorem}{theorems}
\crefformat{theorem}{#2theorem~#1#3} 
\Crefformat{theorem}{#2Theorem~#1#3} 

\crefname{theoremN}{theorem}{theorems}
\crefformat{theoremN}{#2theorem~#1#3} 
\Crefformat{theoremN}{#2Theorem~#1#3} 

\crefname{enumi}{}{}
\crefformat{enumi}{#2#1#3}
\Crefformat{enumi}{#2#1#3}

\crefname{assumption}{assumption}{Assumptions}
\crefformat{assumption}{#2assumption~#1#3} 
\Crefformat{assumption}{#2Assumption~#1#3} 

\crefname{construction}{construction}{Constructions}
\crefformat{construction}{#2construction~#1#3} 
\Crefformat{construction}{#2Construction~#1#3} 

\crefname{sketch}{sketch}{Sketches}
\crefformat{sketch}{#2sketch~#1#3} 
\Crefformat{sketch}{#2Sketch~#1#3} 

\crefname{recollection}{recollection}{Recollections}
\crefformat{recollection}{#2recollection~#1#3} 
\Crefformat{recollection}{#2Recollection~#1#3} 

\crefname{question}{question}{Questions}
\crefformat{question}{#2question~#1#3} 
\Crefformat{question}{#2Question~#1#3} 

\crefname{equation}{}{}
\crefformat{equation}{(#2#1#3)} 
\Crefformat{equation}{(#2#1#3)}

\numberwithin{equation}{section}

\theoremstyle{nonumberplain}
\theoremsymbol{\ensuremath{\blacksquare}}

\newtheorem{proof}{Proof}

\newcommand\bA{{\mathbb A}}

\newcommand\bG{{\mathbb G}}
\newcommand\bH{{\mathbb H}}

\newcommand\bK{{\mathbb K}}

\newcommand\bP{{\mathbb P}}

\newcommand\bR{{\mathbb R}}
\newcommand\bS{{\mathbb S}}
\newcommand\bT{{\mathbb T}}
\newcommand\bU{{\mathbb U}}

\newcommand\bZ{{\mathbb Z}}

\newcommand\cC{{\mathcal C}}

\newcommand\cE{{\mathcal E}}
\newcommand\cF{{\mathcal F}}

\newcommand\ft{{\mathfrak t}}

\DeclareMathOperator{\ind}{\mathrm{ind}}

\newcommand{\cat}[1]{\textsc{#1}}

\newcommand{\comment}[1]{}

\newcommand{\xrightarrowdbl}[2][]{%
  \xrightarrow[#1]{#2}\mathrel{\mkern-14mu}\rightarrow
}

\title{Action principality as a Lie-group certificate}
\author{Alexandru Chirvasitu}

\begin{document}

\date{}

\newcommand{\Addresses}{{
  \bigskip
  \footnotesize

  \textsc{Department of Mathematics, University at Buffalo}
  \par\nopagebreak
  \textsc{Buffalo, NY 14260-2900, USA}  
  \par\nopagebreak
  \textit{E-mail address}: \texttt{achirvas@buffalo.edu}

}}

\maketitle

\begin{abstract}  
  A continuous action $\mathbb{G}\circlearrowright X$ of a topological group is principal if its isotropy groups are all conjugate to $\mathbb{H}\le \mathbb{G}$ and the quotient map $X\to X/\mathbb{G}$ is a locally trivial $\mathbb{G}/\mathbb{H}$-fiber bundle. We prove that compact groups whose identity component has metrizable abelianization are Lie provided their free actions on Tychonoff (equivalently, compact Hausdorff) spaces are all principal; this is a converse to Gleason's theorem. A variant confirms the conclusion for Tychonoff or compact Hausdorff actions with constant central isotropy by compact connected groups. 
\end{abstract}


\vspace{.5cm}

\noindent{MSC 2020: 55R10; 22C05; 54H15; 22D05; 22B05; 55R15; 55R35; 22E15
  

}


\section*{Introduction}

A familiar aspect of compact-Lie-group dynamics is the manageable local orbit/isotropy behavior, manifest in several forms in the literature:
\begin{enumerate}[(a),wide]
\item There is the celebrated Gleason theorem (\cite[Theorem 3.6]{MR33830}, \cite[Theorem 10.34]{hm5} for locally compact spaces, \cite[Theorem II.5.8]{bred_cpct-transf}, etc.) to the effect that actions $\bG\circlearrowright X$ by compact Lie groups on \emph{Tychonoff} (i.e. $T_{3\frac 12}$) spaces of constant \emph{orbit type} $(\bH)$ for $\bH\le \bG$ (or: with all isotropy groups $\bG_x$ $\bG$-conjugate to $\bH\le \bG$) yield locally trivial $\bG/\bH$-fiber bundles $X\to X/\bG$. 

\item\label{item:intro.tubes} Intimately related to this, $T_{3\frac 12}$-actions $\bG\circlearrowright X$ all admit \emph{tubes} (\cite[Theorem II.5.4]{bred_cpct-transf}, \cite[Theorem I.5.7]{td_transf-gp}) surrounding arbitrary orbits $\bG x$, $x\in X$: open $\bG$-invariant neighborhoods of the form $\bG\times_{\bG_x}S$ for a $\bG_x$-space $S$ (a \emph{slice} of the action).

\item Or again: given an action $\bU\circlearrowright \bG$ of one compact Lie group on another and a $\bU$-equivariant \emph{principal $\bG$-bundle} \cite[\S I.8]{td_transf-gp} $\cE\xrightarrowdbl{} X$ over a $T_{3\frac 12}$ space assumed locally trivial as a plain $\bG$-bundle, $\bU$-equivariant local triviality is automatic \cite[Proposition I.8.10]{td_transf-gp}. 
\end{enumerate}

The broad theme underpinning the present note is the extent to which such properties and analogues thereof \emph{characterize} Lie groups among compact groups. A case in point is \Cref{item:intro.tubes}: the compact groups admitting tubes around all of their orbits for arbitrary actions on arbitrary Tychonoff (or equivalently, compact Hausdorff) spaces must be Lie \cite[Theorem 2.1]{2402.08121v2}. 

The results below are very much in the same circle of ideas, but focusing on actions $\bG\circlearrowright X$ with ``controlled'' isotropy: the stabilizer (or isotropy) groups
\begin{equation*}
  \bG_x
  :=
  \left\{g\in \bG\ :\ gx=x\right\}
\end{equation*}
are required to be central, say, or trivial (the latter meaning the action would be free). A slightly paraphrased amalgam of \Cref{th:mtrz.ab,th:lie.iff.cntrl.princ} reads as follows.

\begin{theoremN}\label{thn:intro.cntrl.free}
  Let $\bG$ be a compact group and $\bG_0\le \bG$ its identity connected component.
  \begin{enumerate}[(1)]
  \item If $\bG$ has metrizable connected abelianization $\bG_0/\bG_0'$ the following conditions are equivalent.
    \begin{enumerate}[(a),wide]
    \item $\bG$ is Lie.
    \item All free $\bG$-actions $\bG\circlearrowright X$ on Tychonoff spaces are \emph{($\bG$-)principal}, meaning that $X\xrightarrowdbl{} X/\bG$ is a locally trivial principal $\bG$-bundle.

    \item All free $\bG$-actions on compact Hausdorff spaces are principal. 
    \end{enumerate}

  \item If $\bG$ is connected the following conditions are equivalent.
    \begin{enumerate}[(a),wide]
    \item $\bG$ is Lie.

    \item All $\bG$-actions $\bG\circlearrowright X$ on Tychonoff spaces with constant central isotropy $\bH\le Z(\bG)\le \bG$ are $\bG/\bH$-principal.

    \item All actions on compact Hausdorff spaces with constant central isotropy are principal.
    \end{enumerate}
  \end{enumerate}
\end{theoremN}


\section{Controlled-isotropy/principal coincidence as a Lie characterization}\label{se:fr.princ}

Recall \cite[post Corollary I.4.4]{bred_cpct-transf} that an \emph{orbit type} of an action $\bG\circlearrowright X$ is a conjugacy class $(\bG_x)$ of an isotropy group
\begin{equation*}
  \bG_x:=\left\{g\in \bG\ :\ gx=x\right\}
  ,\quad
  x\in X.
\end{equation*}
A \emph{principal} action $\bG\circlearrowright X$ on a \emph{Tychonoff} (or $T_{3\frac 12}$ \cite[Definition 14.8]{wil_top}) space will be one with constant orbit type $\bH$, with the property that $X\xrightarrowdbl{}X/\bG$ is a locally trivial fibration with fiber $\bG/\bH$. In particular, \emph{free} principal actions produce principal $\bG$-bundles in the usual sense \cite[\S I.8]{td_transf-gp}.

\Cref{th:lie.iff.cntrl.princ} characterizes Lie groups among arbitrary compact connected groups in terms of action principality, very much in the spirit of \cite[Theorem 2.1]{2402.08121v2}.

\begin{remark}\label{re:not.all.can.reps}
  Two distinct mechanisms for the failure of non-Lie compact-group free actions to be principal are visible in \cite[Example 2.2 and Lemma 2.4]{2402.08121v2} respectively:
  \begin{enumerate}[(a),wide]    
  \item\label{item:re:not.all.can.reps:join} In the first instance, an infinite product $\prod_i \bG_i$ of non-trivial compact Lie groups can operate freely on $\prod_i X_i$, with actions $\bG_i\circlearrowright X_i$ of $i$-unbounded \emph{$\bG_i$-index} in the sense of \cite[Definition 6.2.3]{mat_bu} (\emph{co-index} in \cite[\S 3]{cf1}):
    \begin{equation}\label{eq:g.ind}
      \mathrm{ind}_{\bG}\left(X\right)
      =
      \mathrm{ind}_{\bG}\left(\bG\circlearrowright X\right)
      :=
      \inf\left\{n\in \bZ_{\ge 0}\ :\ \exists\ X\xrightarrow[\quad\text{equivariant}\quad]{\quad}E_n\bG:=\bG^{*(n+1)}\right\}
    \end{equation}
    (with ``$*$'' denoting \emph{joins}: \cite[Definition 4.2.1]{mat_bu}, \cite[post Theorem I.6.6]{td_transf-gp}); indeed, one can take $\bG_i\circlearrowright X_i:=\bG_i\circlearrowright E_{n_i}\bG_i$ for unbounded $\left\{n_i\right\}_i\subseteq \bZ_{\ge 0}$.     

  \item\label{item:re:not.all.can.reps:profin} Secondly, profinite groups are finite precisely when their \emph{canonical embeddings}
    \begin{equation}\label{eq:can.emb}
      \bG
      \lhook\joinrel\xrightarrow{\quad\cat{can}=\cat{can}_{\bG}:=\prod_{\rho}\rho\quad}
      \prod_{\rho}\bU(\dim(\rho))
    \end{equation}
    are principal, with $\rho$ ranging over irreducible representations. 
  \end{enumerate}
  It is perhaps worth noting, in context, that these procedures will not \emph{invariably} produce pathological examples: they cannot generally be mixed, as the simple observation in \Cref{le:tor.prod} confirms. 
\end{remark}

\begin{lemma}\label{le:tor.prod}
  For a torus $\bG\cong \bT^I:=\left(\bS^1\right)^I$ (finite-dimensional or not) the canonical embedding $\cat{can}_{\bG}$ of \Cref{eq:can.emb} induces a free translation action. 
\end{lemma}
\begin{proof}
  Simply observe that $\bG$ itself, with its usual translation action, will be a factor in \Cref{eq:can.emb}'s right-hand side. For arbitrary $\bG$-spaces $X$ we have an identification
  \begin{equation*}
    \left(\bG\times X,\ \text{left-translation action}\right)
    \ni
    (g,x)
    \xmapsto{\quad}
    (g,gx)
    \in
    \left(\bG\times X,\ \text{diagonal action}\right),
  \end{equation*}
  hence the conclusion. 
\end{proof}

Some streamlined language will be of use.

\begin{definition}\label{def:prnc.gp}
  Let $\bG$ be a compact group and $\cF=\{(\bH_i)\}_i$ a class of $\bG$-orbit types. \emph{$\cF$-actions} are those with constant orbit type in $\cF$.
  \begin{enumerate}[(1),wide]    
  \item For a class $\cC$ of constant-orbit-type actions on $T_{3\frac 12}$ spaces, $\bG$ is \emph{$\cC$-principal$_{\cF}$} if actions in $\cC$ with constant orbit type in $\cF$ are principal.

  \item The terms apply also to classes of $T_{3\frac 12}$ \emph{spaces} rather than actions: in that case \emph{$\cC$-principal$_{\cF}$} means that $\cF$-actions on spaces $X\in \cC$ all are.

  \item Similarly, \emph{principality$_\cF$} means that $\cF$-actions on all $T_{3\frac 12}$ spaces are principal. $T_{2,\kappa}$ denotes the compact-$T_2$ class.
  \end{enumerate}
  Some families $\cF$ of possible interest would be: $\cat{f}=\{(\{1\})\}$ (standing for ``free''), $\cat{all}$=arbitrary subgroups, $\cat{n}$=normal, $\cat{Ab}$=abelian, $\cat{nAb}$=normal abelian, $\cat{c}$=central. 
\end{definition}

A first remark follows, much as in the resolution of the profinite case (\cite[Lemma 2.4]{2402.08121v2} and \Cref{re:not.all.can.reps}\Cref{item:re:not.all.can.reps:profin}). Recall \cite[Definition post Theorem 25.2]{mnk} that a \emph{locally (path-)connected} space is one whose topology has a basis consisting of (respectively path-)connected open sets. We consistently denote identity connected components of groups $\bG$ by $\bG_0$ (following fairly standard notation: e.g. \cite[p.xi]{hm5}).

\begin{proposition}\label{pr:can.emb.fin.comp}
  If the canonical embedding \Cref{eq:can.emb} of a compact group $\bG$ is principal (so a fortiori if $\bG$ is principal$_{\cat{f}}$) then
  \begin{itemize}[wide]
  \item both the connected-component group $\bG/\bG_0$ and the group $\pi_0(\bG)$ of path components are finite and coincide;
  \item $\bG_0$ is path-connected;
  \item and $\bG$ (along with $\bG_0$) is locally connected and locally path-connected. 
  \end{itemize}
\end{proposition}
\begin{proof}
  The product $\bU$ of unitary groups translated by $\cat{can}(\bG)$ is itself locally path-connected, so principality ensures the existence of a $\bG$-invariant neighborhood $V\ni 1\in \bU$, disjoint union of finitely many (path-)connected open sets $V=\bigsqcup_{1}^n V_i$, admitting a $\bG$-equivariant map onto $\bG$ with $V_i$s' images containing open subsets of $\bG$ which cover the latter. It follows that the components and path components of $\bG$ are open (so coincide), hence the first two claims.

  For the third:
  \begin{itemize}[wide]
  \item map an open neighborhood $\cong \bG\times U\cong V$ of a fixed $\bG$-orbit $\bG x\in \bU$ by
    \begin{equation*}
      \bG\times U
      \xrightarrowdbl[\quad\text{open map}\quad]{\quad \text{first projection }\pi_1\quad}
      \bG;
    \end{equation*}

  \item for open $W\subseteq \bG$ the preimage $\pi_1^{-1}W$ will have open connected (path-)components \cite[Theorems 25.3 and 25.4]{mnk} by local (path-)connectedness;

  \item the connected (path-)components of $W=\pi_1 \pi_1^{-1}W$ are then open, yielding respective local (path-)connectedness for $\bG$ by the converse clauses of the same \cite[Theorems 25.3 and 25.4]{mnk}.
  \end{itemize}
\end{proof}

The following remark will be of some (repeated) use.

\begin{lemma}\label{le:if.loc.triv.emb}
  Let $\bH\le \bG$ be a closed embedding of topological groups with $\bH$ acting principally on $\bG$.
  
  The free action $\bH\circlearrowright Y$ is principal provided the induced action $\bG\circlearrowright X:=\bG\times_{\bH}Y$ is. In particular,
  \begin{equation*}
    \bG\text{ ($T_{2,\kappa}$-)principal$_{\cat{f}}$}
    \xRightarrow{\quad}
    \bH\text{ ($T_{2,\kappa}$-)principal$_{\cat{f}}$}.
  \end{equation*}
\end{lemma}
\begin{proof}
  The second claim follows formally from the first. 
  
  For a constant-orbit-type-$(\bK)$ action $\bH\circlearrowright Y$ principality can be phrased as the existence of \emph{tubes} (\cite[post Theorem I.5.6]{td_transf-gp}, \cite[\S II.4]{bred_cpct-transf}) about orbits $\bH y$:
  \begin{itemize}[wide]
  \item an open $\bH$-neighborhood $U\supseteq \bH y$ admitting an $\bH$-equivariant map $U\to \bH/\bK$;
  \item equivalently, an open neighborhood $\bH$-equivariantly identifiable with $\bH\times_{\bK}Z$ for some $\bK$-space $Z$. 
  \end{itemize}
  Under the hypothesis, $\bH$-orbits $\bH y$ in $Y$ have $\bH$-neighborhoods $U$ admitting $\bH$-equivariant maps $U\xrightarrow{\varphi} \bG$. The principality of $\bH\circlearrowright \bG$ further ensures the existence of $\bH$-equivariant
  \begin{equation*}
    \varphi(\bH y) \subseteq V
    \xrightarrow{\quad\psi\quad}
    \bH,
  \end{equation*}
  hence also
  \begin{equation*}
    \varphi^{-1}V
    \xrightarrow{\quad\psi\circ\varphi\quad}
    \bH:
  \end{equation*}
  the sought-after equivariant map to $\bH$, locally defined around the original $\bH y$. 
\end{proof}

\begin{remark}\label{re:ind.res.act}
  Cf. the companion result (\cite[Observation 2.4(3)]{2604.02099v1}, in turn citing \cite[Lemma 2.9]{2404.12748v1}), to the effect that the principality of $\bG\circlearrowright X$ and $\bH\le \bG$ entails that of $\bH\circlearrowright X$. Although that proof is not phrased in tube language, the technique is the same in essence.

  In fact, the two results imply one another:
  \begin{enumerate}[label={},wide]
  \item\textbf{(restriction $\Rightarrow$ induction)} \cite[Lemma 2.9]{2404.12748v1} being assumed, under the hypotheses of \Cref{le:if.loc.triv.emb} $\bH\circlearrowright \bG\times_{\bH}Y$ is principal. So too, then, is the restricted action on the $\bH$-invariant subspace $Y\subseteq \bG\times_{\bH}Y$. 

  \item\textbf{(induction $\Rightarrow$ restriction)} Conversely, take \Cref{le:if.loc.triv.emb} for granted and consider the principal action $\bG\circlearrowright X$. The $\bG$-equivariant isomorphism
    \begin{equation*}
      \bG\times_{\bH}X
      \ni
      [g,x]      
      \xmapsto[\quad\cong\quad]{\quad}
      (g\bH,gx)
      \in
      \left(\bG/\bH\times X,\ \text{diagonal action}\right)
    \end{equation*}
    then makes $\bG\circlearrowright \bG\times_{\bH}X$ principal \cite[Observation 2.4(2)]{2604.02099v1}, hence the principality of $\bH\circlearrowright X$. 
  \end{enumerate}
\end{remark}

Before revisiting the matter of free-action principality, there is a central-isotropy counterpart relatively easy to extract from the material amassed thus far.

\begin{theorem}\label{th:lie.iff.cntrl.princ}
  The following conditions on a connected compact group $\bG$ are equivalent.
  \begin{enumerate}[(a),wide]
  \item\label{item:th:lie.iff.cntrl.princ:lie} $\bG$ is Lie.

  \item\label{item:th:lie.iff.cntrl.princ:t312} $\bG$ is principal$_{\cat{c}}$.

  \item\label{item:th:lie.iff.cntrl.princ:cpctt2} $\bG$ is $T_{2,\kappa}$-principal$_{\cat{c}}$.

  \end{enumerate}
\end{theorem}
\begin{proof}
  \Cref{item:th:lie.iff.cntrl.princ:lie} $\Rightarrow$ \Cref{item:th:lie.iff.cntrl.princ:t312} $\Rightarrow$ \Cref{item:th:lie.iff.cntrl.princ:cpctt2} by Gleason again, so it suffices to concentrate the proof on \Cref{item:th:lie.iff.cntrl.princ:cpctt2} $\Rightarrow$ \Cref{item:th:lie.iff.cntrl.princ:lie}.
 
  There is a central, connected-kernel quotient \cite[Theorem 9.24]{hm5}
  \begin{equation*}
    \bG
    \xrightarrowdbl{\quad\pi\quad}
    \prod_{i\in I}\bS_i,
    \quad
    \bS_i\text{ compact, simple, connected Lie}.
  \end{equation*}
  Note that principality$_{\cat{c}}$ entails principality$_{\cat{f}}$ for arbitrary central quotients, hence for $\bS$ in the present context. It follows from \cite[Example 2.2]{2402.08121v2} that $I$ is finite and $\bS$ Lie. That being the case, $\pi$ is a principal $\bA$-bundle for $\bA:=\ker\pi$ by \cite[Theorem 10.80 and Exercise E10.8]{hm5}. It follows that a compact-$T_2$ action $\bA\circlearrowright Y$ will be principal as soon as $\bG\circlearrowright \bG\times_{\bA} Y$ is (\Cref{le:if.loc.triv.emb}), thus reducing the problem to compact connected $\bA$ in place of $\bG$. We relegate that case to \Cref{pr:lie.iff.cntrl.princ.ab}. 
\end{proof}

\begin{proposition}\label{pr:lie.iff.cntrl.princ.ab}
  \Cref{th:lie.iff.cntrl.princ} holds for \emph{pro-tori}, i.e. \cite[Definition 9.30]{hm5} compact connected abelian groups.
\end{proposition}
\begin{proof}
  The discrete abelian \emph{Pontryagin dual} $\widehat{\bG}$ \cite[Definition 1.22 and Theorem 7.63]{hm5}  of the group in question is torsion-free \cite[Corollary 8.5]{hm5}. Recall the notion of \emph{rank} for such groups \cite[p.45]{kap}:
  \begin{equation*}
    \mathrm{rk}~\widehat{\bG}
    :=
    \sup\left\{\alpha\ :\ \exists\ \bZ^{\oplus \alpha}\lhook\joinrel\xrightarrow[\quad]{\quad\text{embedding}\quad}\widehat{\bG}\right\}.
  \end{equation*}
  We consider two cases.

  \begin{enumerate}[(I),wide]
  \item \textbf{: infinite $\mathrm{rk}~\widehat{\bG}$.} $\bG$ then surjects onto an infinite-dimensional torus, hence the existence of non-principal actions (on products of smooth compact connected manifolds) by \cite[Example 2.2]{2402.08121v2} again.

  \item \textbf{: $n:=\mathrm{rk}~\widehat{\bG}\in \bZ_{\ge 0}$.} $\bG$ will, in that case, surject onto one of the \emph{solenoids} that form the object of \cite[Theorem 9.71]{hm5}, so we can assume $\bG$ itself is one (if not Lie):
    \begin{equation*}
      \bG
      \cong
      \varprojlim
      \left(
        \cdots
        \xrightarrowdbl{\ \bullet^{n_2}\ }
        \bS^1
        \xrightarrowdbl{\ \bullet^{n_1}\ }
        \bS^1
        \xrightarrowdbl{\ \bullet^{n_0}\ }
        \bS^1
      \right)
      ,\quad
      n_k\in \bZ_{>0}.
    \end{equation*}
    The corresponding profinite limit
    \begin{equation*}
      \bP
      \cong
      \varprojlim
      \left(
        \cdots
        \xrightarrowdbl{\ \bullet^{n_2}\ }
        \bZ/n_1n_0
        \xrightarrowdbl{\ \bullet^{n_1}\ }
        \bZ/n_0
        \xrightarrowdbl{\ \bullet^{n_0}\ }
        \{1\}
      \right)
    \end{equation*}
    embeds into $\bG$ with $\bG/\bP\cong \bS^1$, thus giving a fibration by \cite[Exercise E10.8]{hm5} and Szenthe's \cite[Theorem 10.80]{hm5} again. This once more permits the deduction via \Cref{le:if.loc.triv.emb} of $(\bP\circlearrowright Y)$'s principality from that of $\bG\circlearrowright \bG\times_{\bP}Y$, so the problem has been reduced to the profinite groups already disposed of in \cite[Lemma 2.4]{2402.08121v2}.
  \end{enumerate}
\end{proof}

\begin{lemma}\label{le:princ.fin.ind}
  Given a finite-index compact-group embedding $\bH\le \bG$, $\bH$ is ($T_{2,\kappa}$-)principal$_{\cat{f}}$  if and only if $\bG$ is respectively so.
\end{lemma}
\begin{proof}
  On the one hand, $\bG\circlearrowright X$ is principal if and only if $\bH\circlearrowright X$ is. On the other, an arbitrary $\bH\circlearrowright Y$ is principal if and only if the induced $\bG$-action on
  \begin{equation*}
    \bG\times_{\bH}Y
    :=
    \bG\times Y/\left((g,y)\sim (gh,h^{-1}y),\ \forall \left(g\in \bG,\ h\in \bH,\ y\in Y\right)\right)
  \end{equation*}
  (\emph{twisted product} \cite[\S I.6(A)]{bred_cpct-transf}) is. 
\end{proof}

Some further consequences of the above include the following, recorded here for future reference. Recall \cite[Theorem 9.2]{hm5} for context that the algebraic \emph{commutator subgroup}
\begin{equation*}
  \bG'
  :=
  \Braket{[g,h]\ :\ g,h\in \bG}
  ,\quad
  [g,h]
  :=
  ghg^{-1}h^{-1}
\end{equation*}
of a compact \emph{connected} group is automatically closed. 

\begin{proposition}\label{pr:comm.ab}  
  \begin{enumerate}[(1),wide]
  \item\label{item:pr:comm.ab:id.comp} A compact group $\bG$ is ($T_{2,\kappa}$-)principal$_{\cat{f}}$ if and only if its identity (path-)connected component $\bG_0$ is.

  \item\label{item:pr:comm.ab:comm.ab} Moreover, in that case so too are the commutator subgroup $\bG_0'$ and the \emph{abelianization} $\bG_{0,ab}:=\bG_0/\bG_0'$. 
  \end{enumerate}
\end{proposition}
\begin{proof}
  The claims in \Cref{item:pr:comm.ab:id.comp}, implicit or explicit, follow from \Cref{pr:can.emb.fin.comp} and \Cref{le:princ.fin.ind}. We then have \Cref{item:pr:comm.ab:id.comp} $\Rightarrow$ \Cref{item:pr:comm.ab:comm.ab} via \Cref{le:if.loc.triv.emb} and the semidirect-product decomposition \cite[Theorem 9.39(i)]{hm5}
  \begin{equation*}
    \bG_0
    \cong
    \bG_0'\rtimes \bG_{0,ab}
  \end{equation*}
  ensuring the principality (in fact, triviality) of both $\bG_0'\circlearrowright \bG_0$ and $\bG_{0,ab}\circlearrowright \bG_0$.
\end{proof}

Recall \cite[Definition 9.5]{hm5} that a \emph{semisimple} compact connected group is one coinciding with its derived subgroup; we extend the term to arbitrary compact groups to mean that $\bG_0$ is semisimple. 

\begin{theorem}\label{th:mtrz.ab}
  For a compact group $\bG$ with metrizable connected abelianization $\bG_0/\bG_0'$ the following conditions are equivalent.
  \begin{enumerate}[(a),wide]
  \item\label{item:th:mtrz.ab:lie} $\bG$ is Lie. 

  \item\label{item:th:mtrz.ab:t312} $\bG$ is principal$_{\cat{f}}$.

  \item\label{item:th:mtrz.ab:cpctt2} $\bG$ is $T_{2,\kappa}$-principal$_{\cat{f}}$.
  \end{enumerate}
  In particular, the conditions are equivalent for semisimple $\bG$. 
\end{theorem}
\begin{proof}
  \Cref{item:th:mtrz.ab:lie} is the strongest condition per Gleason's theorem, so the substance is rather contained in the converse statements forcing the Lie condition given principality. 
  
  \Cref{pr:comm.ab}\Cref{item:pr:comm.ab:id.comp} reduces the problem to connected $\bG$ (assumption henceforth in place), while \Cref{pr:comm.ab}\Cref{item:pr:comm.ab:comm.ab} provides the $T_{2,\kappa}$-principality$_{\cat{f}}$ of both $\bG'$ and $\bG/\bG'$. The latter is quickly disposed of under the hypothesis: being a metric, connected, locally connected (\Cref{pr:can.emb.fin.comp}) compact abelian group, it must be a torus $(\bS^1)^I$ \cite[Theorem 8.46]{hm5}; its finite-dimensionality then follows from \cite[Example 2.2]{2402.08121v2}.

  All in all, then, we may take $\bG$ (connected and) semisimple. It thus fits into a commutative diagram
  \begin{equation*}
    \begin{tikzpicture}[>=stealth,auto,baseline=(current  bounding  box.center)]
      \path[anchor=base] 
      (0,0) node (l) {$\prod_{i\in I} \bS_i^*$}
      +(3,.5) node (u) {$\bG$}
      +(6,0) node (r) {$\prod_{i\in I} \bS_{i*}$}
      ;
      \draw[->>] (l) to[bend left=6] node[pos=.5,auto] {$\scriptstyle \pi^*$} (u);
      \draw[->>] (u) to[bend left=6] node[pos=.5,auto] {$\scriptstyle \pi_*$} (r);
      \draw[->>] (l) to[bend right=6] node[pos=.5,auto,swap] {$\scriptstyle \text{obvious surjection}$} (r);
    \end{tikzpicture}
  \end{equation*}
  of central-profinite-kernel surjections \cite[Theorem 9.19]{hm5}, where $\bS_i^*$ are simple, connected, simply connected compact Lie groups and
  \begin{equation*}
    \bS^*_i\xrightarrowdbl{\quad}\bS_{i*}:=\bS_i^*/Z(\bS_i^*)
  \end{equation*}
  are their respective \emph{adjoint forms}. The free $T_{2,\kappa}$-action $\bG\circlearrowright X$ witnessing non-principality when $I$ is infinite (postponing some of the choices involved for the moment) will be
  \begin{equation*}
    \alpha:\bG\circlearrowright X:=\left(\prod_{i\in I}X_i\right)/\ker \pi^*,
    \quad
    \forall i
    \left(\alpha_i:\bS_i^*\circlearrowright X_i\text{ free and hence principal}\right).
  \end{equation*}
  Were $\alpha$ principal, it would admit a $\bG$-map $X\to E_n\bG$ for some (any sufficiently large) $n\in \bZ_{>0}$. Embedding
  \begin{equation*}
    X_i
    \ni
    x
    \xmapsto[\quad \text{$\bS^*_i$-equivariantly}\quad]{\quad}
    \left(y_j
      :=
      \begin{cases}
        x&i=j\\
        \text{some fixed $p_j\in X_j$}&i\ne j
      \end{cases}
    \right)_{j\in I}
    \in
    \prod_j X_j,
  \end{equation*}
  this in turn gives
  \begin{itemize}[wide]
  \item an $\bS^*_i$-equivariant map $X_i\to E_n\bG\to E_n\bS_{i*}$;
  \item so that $\mathrm{ind}_{\bS_{i*}}X_i/Z(\bS_i^*)\le n$ (recall \Cref{eq:g.ind}). 
  \end{itemize}
  To force non-principality, then, it suffices to argue that an arbitrary simple simply connected compact Lie group $\bS$ has free actions $\beta:\bS\circlearrowright Y$ with induced free actions $\beta_*:\bS_{*}\circlearrowright Y_*:=Y/\left(Z:=Z(\bS)\right)$ of the adjoint form having arbitrarily large index (for then one could ensure the said indices are unbounded when $I$ is infinite).

  It will be helpful, to that end, to recall the bound
  \begin{equation}\label{eq:ind.ge.ind}
    \ind_{\bS_*}Y_*
    \ge
    \ind_{\mu}Y_*
    \quad
    \text{(consequence of \cite[Proposition 3.11]{MR801933})},
  \end{equation}
  where
  \begin{itemize}[wide]
  \item $\mu\in H^*(B\bS_*,\bR)$ is a class with vanishing $H^0$ component in the real cohomology algebra of the \emph{classifying space} \cite[Definition 7.2.7]{hjjm_bdle} of $\bS_*$ (equivalently \cite[Remark 3.4]{MR801933}, the \emph{equivariant} cohomology algebra $H^*_{\bS_*}(\{*\},\bR)$);
  \item and
    \begin{equation*}
      \ind_{\mu}(Y_*)
      =
      \ind_{\mu}(\beta_*:\bS_*\circlearrowright Y_*)
      :=
      \sup\left\{k\in \bZ_{\ge 0}\ :\ q^*_{\beta_*}\mu^k\ne 0\right\}
    \end{equation*}
    for the \emph{classifying map} \cite[Theorem I.8.12]{td_transf-gp} $Y_*/\bS_*\xrightarrow{q_{\beta_*}}B\bS_*$ attached to the action $\beta_*$ and the principal $\bS_*$-bundle it produces (what \cite[Definition 3.9]{MR801933} specializes to for free actions; see also \cite[\S 3]{MR478189}\footnote{Note the numerical shift: our $\ind_{\mu}$ coincides with \cite[Definition 3.1]{MR478189}, but is \cite[Definition 3.9]{MR801933}'s $\ind_{\mu}-1$.}). 
  \end{itemize}
  For \emph{paracompact} \cite[Definition 20.6]{wil_top} $\bS$-spaces $\beta:\bS\circlearrowright Y$ identify isomorphism classes of resulting bundles $Y\xrightarrowdbl{} Y/\bS$ with
  \begin{equation*}
    [Y/\bS,B\bS]
    :=
    \text{homotopy classes of maps $Y/\bS\to B\bS$}.
  \end{equation*}
  The induced-action/bundle map
  \begin{equation*}
    [Y/\bS,B\bS]
    \ni
    \left(Y\xrightarrowdbl{\quad}Y/\bS\right)
    \xmapsto{\quad}
    \left(Y_*\xrightarrowdbl{\quad}Y_*/\bS_*\right)
    \in
    [Y_*/\bS_*\cong Y/\bS,B\bS_*]
  \end{equation*}
  is nothing but composition with the fibration $B\bS\xrightarrow{\cat{fib}_{\bS}} B\bS_*$ resulting \cite[\S 18.3.6]{hjjm_bdle} from the exact sequence
  \begin{equation*}
    \{1\}
    \to
    Z
    \lhook\joinrel\xrightarrow{\quad}
    \bS
    \xrightarrowdbl{\quad}
    \bS_*
    \to
    \{1\}.
  \end{equation*}
  In particular, the \emph{universal} \cite[Chapter I, (8.11)]{td_transf-gp} bundle $E\bS\xrightarrowdbl{} B\bS$, upon quotienting the total space by $Z$, turns into (the bundle/action attached to) that selfsame canonical fibration $\cat{fib}_{\bS}$.

  Denoting by $\ft:=Lie(\bT)$, $\ft_*:=Lie(\bT_*)$ for maximal tori $\bS\ge \bT\xrightarrowdbl{} \bT_*\le \bS_*$, $W$ for the \emph{Weyl group} \cite[Definition 6.22]{hm5} of $\bS$ and $(\bullet)^*$ for dual vector spaces,
  \begin{equation*}
    H^*(B\bS_*,\bR)
    \cong
    S(\ft_*^*)^W
    \lhook\joinrel\xrightarrow{\quad \cat{fib}^*_{\bS}\quad}
    S(\ft^*)^W
    \cong
    H^*(B\bS,\bR)
    \quad
    \left(\text{\cite[post Theorem 6.8.1]{gs_supersym_1999}}\right)
  \end{equation*}
  is an embedding of polynomial rings by \cite[post Theorem 6.8.3]{gs_supersym_1999}, with ``$W$'' superscripts meaning $W$-invariants and $S(\bullet)$ denoting symmetric algebras. In particular, for \emph{any} non-zero $\mu\in H^{\ge 1}(B\bS_*,\bR)$ we will have $\ind_{\mu}(\beta_*)\gg n\Rightarrow \ind_{\bS_*}(\beta_*)\gg 0$ if $\beta: \bS\circlearrowright Y$ is the truncated universal free action $\bS\circlearrowright E_m \bS$ for sufficiently large $m$. 
\end{proof}

\begin{remark}\label{re:ind.no.free.term}
  \Cref{eq:ind.ge.ind} is invalid, as are \cite[Proposition 3.7]{MR478189} and \cite[Proposition 3.11 $\overline{P}_1$]{MR801933} (both essentially to the effect that $\ind_{\mu}(\text{left-hand translation }\bG\circlearrowright\bG\times Z)=0$), without the stipulation that $\mu\in H^{\ge 1}$ (i.e. that it have vanishing free term). It is unclear whether that requirement is regarded as implicit in the phrase ``characteristic class'' employed in both cited sources in referring to $\mu$ (there denoted by $\alpha$). 
\end{remark}


\addcontentsline{toc}{section}{References}

\def\polhk#1{\setbox0=\hbox{#1}{\ooalign{\hidewidth
  \lower1.5ex\hbox{`}\hidewidth\crcr\unhbox0}}}


\Addresses

\end{document}